\documentclass[12pt]{amsart}
\usepackage{amsthm}
\usepackage{amssymb}
\usepackage{amsmath}
\usepackage{color}
\usepackage{verbatim}
\theoremstyle{plain}
\newtheorem{proposition}{Proposition}
\newtheorem{theorem}[proposition]{Theorem}
\newtheorem{lemma}[proposition]{Lemma}
\newtheorem{corollary}[proposition]{Corollary}

\theoremstyle{definition}

\theoremstyle{definition}

\newtheorem{remark}[proposition]{Remark}

\numberwithin{equation}{section}

\numberwithin{proposition}{section}

\gdef\myletter{}

\let\savetheequation\theequation
\def\theequation{\savetheequation\myletter}

\usepackage{color}

\newcommand{\CC}{{\mathbb C}}

\newcommand{\RR}{{\mathbb R}}

\newcommand{\PP}{{\mathbb P}}

\newcommand{\NN}{{\mathbb N}}

	\renewcommand{\date}{\today}
	\def\bar{\overline}
	\def\hat{\widehat}
	
	\usepackage{xcolor}

	\makeatletter\def\SL@eqnlefttext #1{\hbox to 0pt{\kern 60pt 
			\llap{\SL@margintext{#1}\quad}\hss}}
	\makeatother
	
\begin{document}

		\title[Random Weighted]{\bf Random Sums of Weighted Orthogonal Polynomials in $\CC^d$}
		
		\author{T. Bloom, D. Dauvergne and N. Levenberg*}
		\thanks{*Supported in part by Simons grant}
		
		\address{University of Toronto, Toronto, Ontario M5S 2E4 Canada}  
		\email{bloom@math.toronto.edu}
		
		\address{University of Toronto, Toronto, Ontario M5S 2E4 Canada}  
		\email{duncan.dauvergne@utoronto.ca}
		
		\address{Indiana University, Bloomington, IN 47405 USA}
		\email{nlevenbe@iu.edu}
		
		\begin{abstract}
We consider random polynomials of the form $G_n(z):= \sum_{|\alpha|\leq n} \xi^{(n)}_{\alpha}p_{n,\alpha}(z)$ where $\{\xi^{(n)}_{\alpha}\}_{|\alpha|\leq n}$ are i.i.d. (complex) random variables and $\{p_{n,\alpha}\}_{|\alpha|\leq n}$ form a basis for $\mathcal P_n$, the holomorphic polynomials of degree at most $n$ in $\CC^d$. In particular, this includes the setting where $\{p_{n,\alpha}\}$ are orthonormal in a space $L^2(e^{-2n Q} \tau)$, where $\tau$ is a compactly supported Bernstein-Markov measure and $Q$ is a continuous weight function. Under an optimal moment condition on the random variables $\{\xi^{(n)}_{\alpha}\}$, in dimension $d=1$ we prove convergence in probability of the zero measure to the weighted equilibrium measure, and in dimension $d \ge 2$ we prove convergence of zero currents.
		\end{abstract}

		\maketitle
		
		\section{Introduction} We consider random polynomials in $\CC^d, \ d\geq 1$, of degree at most $n$ of the form 
		\begin{equation}
			\label{probhb2}
			G_n(z):=\sum_{|\alpha|\leq n} \xi^{(n)}_{\alpha}p_{n,\alpha}(z)
			\end{equation}
		where $\{\xi^{(n)}_{\alpha}\}_{|\alpha|\leq n}$ is an array of i.i.d. (complex) random variables and the set $\{p_{n,\alpha}\}_{|\alpha|\leq n}$ forms a basis for $\mathcal P_n$, the holomorphic polynomials of degree at most $n$ in $\CC^d$. Here
		$$
		p_{n,\alpha}(z)=a_{n,\alpha}z^{\alpha}+\sum_{\beta \prec \alpha} c_{n,\beta}z^{\beta}=: a_{n,\alpha}\tilde p_{n,\alpha}(z)
		$$
		where $z^{\alpha}:=z_1^{\alpha_1} \cdots z_d^{\alpha_d}$ and $\prec$ denotes an ordering on the set $\NN^d = \{0, 1, \dots, \}^d$ of multi-indices $\alpha$ which satisfies $\alpha \prec \beta$ whenever $ |\alpha| := \alpha_1 + \cdots + \alpha_d < |\beta|$, and if $\alpha \prec \beta$, then $\alpha + e_j \prec \beta + e_j$ for any coordinate vector $e_j:=(0,...,1,...,0)$. For example, if $\alpha \prec \beta$ for $|\alpha| < |\beta|$ and $\prec$ is the standard lexicographic order on each of the sets $\{\alpha : |\alpha| = n\}$, then $\prec$ is a valid ordering. We will sometimes write $\{p_{n,\alpha}\}_{|\alpha|\leq n}=\{p_{n,j}\}_{j=1,...,m_n}$ where $m_n:=\binom{d+n}{d}=$dim$(\mathcal P_n$).
		
		The random variables $\xi^{(n)}_{\alpha}$ are defined on a probability space $(\Omega, \mathcal H, \mathbb P)$, and we will sometimes write $G_n(z,\omega),  \ \omega \in \Omega$.
		We always assume that the random variables are non-degenerate, i.e., the law of $\xi^{(n)}_{\alpha}$ is not supported on a single point. We are interested in studying the distribution of the zero set of $G_n$ as $n \to \infty$. 
		
		We extend some of the results in \cite{ba}, \cite{Bay}, \cite{blrp} and \cite{BBL} on asymptotics of normalized zero measures and currents associated to $G_n$. All our proofs and theorems will apply to random polynomials in $\CC^d$ for arbitrary $d \ge 1$, however we introduce the problems first in dimension one, where the problems are more classical and the results have a simpler interpretation in terms of classical potential theory.
		
		\subsection{Dimension one.} \qquad 
		The simplest example of the form \eqref{probhb2} is the Kac ensemble in $d=1$, given by
		$$
		H_n(z) = \sum_{j=0}^n \xi_j z^j.
		$$
		In the case when the $\xi_j$ are i.i.d.\ complex Gaussian random variables of mean $0$ and variance $1$, Hammersley \cite{H} showed that the zeros of $H_n$ concentrate on the unit circle. As one might expect, the asymptotic distribution of the zeros only depends on the $\xi_j$ through a moment condition. To this end, Ibragimov and Zaporozhets \cite{I} showed that the condition
		\begin{equation}
		\label{E:log11}
	\mathbb E(\log(1 + |\xi_0|)) < \infty
		\end{equation}
		is both necessary and sufficient for almost sure weak-$*$ convergence of the \textit{zero measure}  $\tfrac{1}{n} \sum_{i=1}^n \delta_{z_j}$, where $z_1,...,z_n$ are the zeros of $H_n$, to normalized acrlength measure $\tfrac{1}{2\pi} d \theta$ on the unit circle. 
		
		If instead we are interested only in convergence in probability of the zero measure, then the correct necessary and sufficient condition is a slightly weakened version of \eqref{E:log11}:
		\begin{equation}
			\label{E:log12}
			\mathbb P(\log(1 + |\xi_0|) > t) = o(t^{-1}).
		\end{equation}
		The necessity and sufficiency of this condition was shown in \cite{BD}. 
		
		A useful way to think about these results is through the lens of orthogonal polynomials and potential theory. From this point of view, we can think of $\{z^n, n \in \mathbb N\}$ as orthonormal polynomials in the space $L^2(\tau), \ \tau = \tfrac{1}{2\pi} d \theta$, and view these convergence results as saying that the zero measure of a random sum of orthonormal polynomials converges under condition \eqref{E:log12} to the equilibrium measure $\mu_K$ on the support $K$ of $\tau$, i.e.\ the unique probability measure supported on $K$ minimizing the logarithmic energy
		$$
		I(\mu) := \int \int \log \frac{1}{|z-t|} d \mu(z) d \mu(t).
		$$
		In other words, we can approximate the equilibrium measure by the zero measure of a \emph{random} sum of orthonormal polynomials.	With this in mind, a natural generalization of the Kac ensemble is to take the polynomial basis $p_{n, \alpha}$ in \eqref{probhb2} to be given by a more general collection of orthonormal polynomials. More precisely, in this work we allow $\{p_{n, j} : j = 0, \dots, n\}$ to be
		\begin{enumerate}
			\item orthonormal polynomials in $L^2(e^{-2nQ}\tau)$ for $\mathcal P_n$ where $p_{n, j}$ has degree $j$. Here $\tau$ is a probability measure supported on a regular compact set $K$, and $Q:K \to \mathbb R$ is a continuous real-valued weight function (external field). We impose the mild regularity assumption that $\tau$ is a {\it weighted Bernstein-Markov measure for $K,Q$}, or more generally
			\item an orthonormal basis in $L^2(e^{-2nQ}\tau_n)$ for $\mathcal P_n$ associated to a sequence of {\it asymptotically weighted Bernstein-Markov measures $\tau_n$ for $K,Q$}. 
		\end{enumerate}
		See Section \ref{S:prelim} for precise definitions. The weight $e^{-2 n Q}$ is chosen so that in the limit, we expect the zero measure $\mu_n$ of the random sum \eqref{probhb2} to converge to the weighted equilibrium measure $\mu_{K, Q}$, i.e.\ the probability measure minimizing the weighted energy
		$$
		I(\mu) + 2 \int Q d \mu
		$$
		among all probability measures supported on $K$. In the unweighted case $Q \equiv 0$, \cite{BD} showed that under assumption (1) above and condition \eqref{E:log12} on the random variables, $\mu_n$ converges in probability to the equilibrium measure in the weak* topology. 
		
		The behaviour of the $G_n$ when $Q \not\equiv 0$ is quite different. For example, when $Q \equiv 0$ the equilibrium measure $\mu$ is always supported on the outer boundary of the set $K$, whereas general weighted equilibrium measures can be quite arbitrary (e.g., see \cite[Section IV]{ST}). Moreover, when $Q \equiv 0$, the array $\{p_{n, i} : i \le n\}$ is simply a sequence $p_{n, i} = p_i$ of orthonormal polynomials, a point which is crucial to the proofs in \cite{BD, DD}.

		In the case of a non-zero external field, \cite{BBL} showed that $\mu_n$ converges weakly to $\mu$ under assumption \eqref{E:log12} whenever the distribution of the $\xi_j$ has no atoms. Our main theorem in dimension one removes this condition. 
		
		\begin{theorem} \label{T:mainth-dim1} Let $\{\xi^{(n)}_i, i \le n \in \mathbb N\}$ be an array of i.i.d. random variables satisfying \eqref{E:log12}; let $K\subset \CC^d$ be compact and regular; and let $Q$ be a continuous, real-valued function on $K$. Let $\{p_{n, i}\}_{i \le n}$ be a basis for $\mathcal P_n$ as in (1) or (2). For random polynomials $G_n$ as in \eqref{probhb2} with zero set $\{z_1, \dots, z_n\}$ we have
			$$
			\frac{1}{n} \sum_{i=1}^n \delta_{z_i} \to \mu_{K, Q}
			$$
			in probability in the weak* topology on measures on $\mathbb C$.
		\end{theorem}
	\begin{remark}
	Theorem \ref{T:mainth-dim1} is the strongest possible result. Indeed the method introduced in \cite{DD} for showing necessity of the condition \eqref{E:log12} in the unweighted case $Q \equiv 0$ also works in the weighted case with a few minor modifications. For brevity we do not pursue this direction here. 
	\end{remark}
		
	\subsection{Dimension $d \ge 2$.} \qquad In dimensions $d \ge 2$, we view the zero set of $G_n$ as a $(1, 1)$-current (see \cite{klimek}):
	$$
	Z_{G_n} := dd^c \left( \frac{1}{\deg(G_n)} \log |G_n| \right)
	$$
	where $dd^c = \tfrac{i}{\pi} \partial \bar \partial$.  When paired with a test form $\phi$ (i.e., \ a smooth $(n-1,n-1)$-form with compact support), $\langle Z_{G_n}, \phi \rangle \in \CC$. We endow the space of $(1, 1)$-currents with the weak* topology and call $Z_{G_n}$ the \textbf{zero current} of $G_n$. In dimension $1$, $dd^c = \frac{1}{2\pi} \Delta$, where $\Delta$ is the Laplace operator, and this definition gives back the zero measure. 
	The operator $dd^c$ is continuous on $L^1_{\operatorname{loc}}(\mathbb C)$, so it is natural to study convergence of the functions $\frac{1}{\deg(G_n)} \log |G_n|$ there as $L^1_{\operatorname{loc}}(\mathbb C^d)$ is a more concrete space than the space of all $(1, 1)$-currents. In particular, it is metrizable, which allows us to consider convergence in probability on this space and then pass results down to the (non-metrizable) space of $(1, 1)$-currents.
	
	As in the one-variable case, we consider orthonormal polynomial bases of the form (1) or (2) above (with $\{p_{n, j} : j = 1, \dots, m_n\}$), but we require that the compact set $K$ be locally regular, rather than regular. In dimension $1$, these notions coincide. The correct multivariate analogue of condition \eqref{E:log12} is
	\begin{equation}
		\label{E:log13-multi}
		\mathbb P(\log(1 + |\xi_\alpha^{(n)}|) > t) = o(t^{-d}).
	\end{equation}
	Indeed, \cite{BD} showed that for polynomial bases of the form (1) above with $Q \equiv 0$, if $G_n$ is built from i.i.d.\ random variables satisfying \eqref{E:log13-multi}, then the functions $\tfrac{1}{\deg(G_n)} \log |G_n|$ converge in probability. As in the one-variable setting, \cite{BBL} extended this to all ensembles satisfying (1) or (2) with the additional assumption that the distribution of the $\xi_\alpha^{(n)}$ has no atoms, and we remove this condition in the present work.
	
	\begin{theorem} \label{mainthd} Let $\{\xi^{(n)}_{\alpha}\}$ be an array of i.i.d. random variables satisfying \eqref{E:log13-multi}; let $K\subset \CC^d$ be compact and locally regular; and let $Q$ be a continuous, real-valued function on $K$. Let $\{p_{n,\alpha}\}_{|\alpha|\leq n}$ be a basis for $\mathcal P_n$ as in (1) or (2). For random polynomials $G_n$ as in (\ref{probhb2}), 
		$$\frac{1}{n}\log |G_n|\to V_{K,Q} \ \hbox{in} \ L^1_{loc}(\CC^d)$$
		in probability.
	\end{theorem}
	We will give the definitions of all notions in Theorem \ref{mainthd} in the next section. Note that despite the similarity to the main result of \cite{BBL}, the potential for distributions with atoms requires a completely different method of proof.
	
	The next corollary translates Theorem \ref{mainthd} to a statement about convergence of zero measures. Note that both of the types of convergence we state below are equivalent to convergence in probability in metrizable spaces.
	\begin{corollary}
	\label{C:zero-currents}
	In the setting of Theorem \ref{mainthd}, the zero currents $Z_{G_n}$ converge to $dd^c V_{K, Q}$ in either of the following senses:
	\begin{itemize}
		\item For any sequence $Y \subset \mathbb N$ there is a further subsequence $Y' \subset Y$ such that $Z_{G_n} \to dd^c V_{K, Q}$ weak*-almost surely along $Y'$.
		\item For any weak*-open set $O$ containing $dd^c V_{K, Q}$,
		$$
		\lim_{n \to \infty} \mathbb P(Z_{G_n} \in O) = 1.
		$$
	\end{itemize}
	\end{corollary}
		
		In the next section, we provide the necessary (weighted) potential- and pluripotential-theoretic background, including a key result on the leading coefficients $a_{n,\alpha}$ of $p_{n,\alpha}$ (Proposition \ref{03}). Section 3 outlines the general strategy of the proof of Theorem \ref{mainthd}, reducing it to verifying the crucial Proposition \ref{01}. The proof of this result is completely different to its analogue in \cite{BBL}; the key step is Theorem \ref{02}. Theorem \ref{T:mainth-dim1} and Corollary \ref{C:zero-currents} are also deduced from Theorem \ref{mainthd} in this section. The details of the actual proof of Theorem \ref{02} are given in section 4.

		\section{Weighted (pluri-)potential theory preliminaries}
		\label{S:prelim}
		We begin with some general definitions, valid in $\CC^d$ for any $d\geq 1$. For $K \subset \CC^d$ compact, let $V_K^*(z):=\limsup_{\zeta \to z}V_K(\zeta)$ where 
		\begin{align*}
V_K(z):&=\sup\{u(z):u\in L(\CC^d) \ \hbox{and} \  u\leq 0 \ \hbox{on} \ K\} \\
&= \sup \{{1\over {\rm deg} (p)}\log |p(z)|:p\in \bigcup {\mathcal P}_n, \  \|p\|_{K}:=\sup_{\zeta \in K}|p(\zeta)| \leq 1 \}.
		\end{align*}
		Here $L(\CC^d)$ is the set of all plurisubharmonic functions on $\CC^d$ of logarithmic growth: $u\in L(\CC^d)$ if $u$ is plurisubharmonic in $\CC^d$ and $u(z)\leq \log |z| +O(1)$ as $|z|\to \infty$. Note that $V_K=V_{\hat K}$ where
		\begin{equation} \label{phull}  \hat K:=\{z\in \CC^d: |p(z)|\leq ||p||_K \ \hbox{for all polynomials} \ p\}. \end{equation}
		 It is known that either $V_K^*\equiv +\infty$ ($K$ is pluripolar) or $V_K^*\in L(\CC^d)$; in this paper, we always assume the latter holds. The compact set $K\subset \CC^d$ is {\it regular} at $a\in K$ if $V_K$ is continuous at $a$; and $K$ is regular if it is continuous on $K$, i.e.,\ if $V_K=V_K^*$. We call a compact set $K\subset \CC^d$ {\it locally regular} if for any $z_0 \in K$, $V_{K\cap \bar{B(z_0,r)}}$ is continuous at $z_0$ for $r>0$ where $B(z_0,r):=\{z:|z-z_0|< r\}$; this is an a priori stronger notion than regularity but these notions coincide if $d=1$. Let $Q$ be a continuous, real-valued function on $K$ and let $w:=e^{-Q}$. The weighted extremal function for $K,Q$ is defined as
		\begin{align*}
	V_{K,Q}(z):&=\sup\{u(z): u\in L(\CC^d), \ u\leq Q \ \hbox{on} \ K\} \\
	&= \sup \{\frac{1}{\deg(p)}\log |p(z)|: p\in \cup_n \mathcal P_n, \ \|pe^{-\deg(p)\cdot Q}\|_K\leq 1\}.
		\end{align*}
		It is known that for $K$ locally regular and $Q$ continuous, $V_{K,Q}$ is continuous on $\CC^d$ (cf., \cite{NQD}).

		To explain the allowable bases $\{p_{n,j}\}_{j=1,...,m_n}$ in Theorems \ref{T:mainth-dim1} and \ref{mainthd}, given a positive measure $\tau$ on $K$, we say that $(K,Q,\tau)$ satisfies a weighted Bernstein-Markov property if 
		\begin{equation} \label{wbwin} \|w^np_n\|_K\leq M_n \|w^np_n\|_{L^2(\tau)} \ \hbox{for all} \ p_n \in \mathcal P_n, \ n=1,2,...\end{equation}
		where $\lim_{n\to \infty} M_n^{1/n}=1$. For such triples, letting $\{p_{n,j}\}_{j=1,...,m_n}$ be an orthonormal basis in $L^2(e^{-2nQ}\tau)$ for $\mathcal P_n$, define
		\begin{equation} \label{bndef} B_n(z):=\sum_{j=1}^{m_n}|p_{n,j}(z)|^2.\end{equation}
		It is well-known that 
		$$\lim_{n\to \infty} \frac{1}{2n}\log B_n(z)=V_{K,Q} \ \hbox{locally uniformly on} \ \CC^d$$
		and this is the key deterministic property of our bases $\{p_{n,j}\}$ that we will need. More generally, let $\{\mu_n\}$ be a sequence of probability measures on $K$ with the following property: if $M_n$ is the smallest constant such that 
		$$\|qe^{-nQ}\|_K\leq M_n \|q\|_{L^2(e^{-2nQ}\mu_n)}=M_n\|qe^{-nQ}\|_{L^2(\mu_n)}   \ \hbox{for all} \ q \in \mathcal P_n,$$
		then $\lim_{n\to \infty} M_n^{1/n}=1$. We say that $\{\mu_n\}$ are {\it asymptotically weighted Bernstein-Markov} for $K,Q$ (simply asymptotically Bernstein-Markov if $Q\equiv 0$). In \cite[Proposition 2.8]{BBL} the following was proved (it was stated slightly differently there):
		
		\begin{proposition}\label{locunif} Let $K\subset \CC^d$ be a locally regular compact set and $Q$ a continuous, real-valued function on $K$. Let $\{\mu_n\}$ be an asymptotically weighted Bernstein-Markov sequence for $K,Q$ and let $\{p_{n,j}\}_{j=1,...,m_n}$ be an orthonormal basis  for $\mathcal P_n$ in $L^2(e^{-2nQ}\mu_n)$. Then for $B_n$ as in (\ref{bndef}),
			$$\lim_{n\to \infty} \frac{1}{2n}\log B_n=V_{K,Q} \ \hbox{locally uniformly on} \ \CC^d.$$
		\end{proposition}
		
		Now let $K\subset \CC^d$ be compact and locally regular. For $Q$ continuous on $K$, we have $w=e^{-Q}>0$ on $K$.
		Fix $\tau$ a finite measure on $K$ such that $(K,Q,\tau)$ is weighted Bernstein-Markov. Let 
		$$\mathcal P(\alpha)=\mathcal P_d(\alpha):=\{z^{\alpha}+\sum_{\beta \prec \alpha} c_{\beta}z^{\beta}: c_{\beta}\in \CC\}.$$
		Recall that our ordering $\prec$ on multiindices satisfies $\beta \prec \alpha$ whenever $|\beta| < |\alpha|$. We write
		\begin{equation}
			\label{pnal}
p_{n,\alpha}(z)=a_{n,\alpha}z^{\alpha}+\sum_{\beta \prec \alpha} c_{n,\beta}z^{\beta}=:a_{n,\alpha}\tilde p_{n,\alpha}(z)
		\end{equation}
		where $\{p_{n,\alpha}\}_{|\alpha|\leq n}$ is an orthonormal basis in $L^2(e^{-2nQ}\tau)$ for $\mathcal P_n$. Then
		$$\tilde p_{n,\alpha}(z)=z^{\alpha}+...\in \mathcal P(\alpha)$$
		and $\|w^np_{n,\alpha}\|_{L^2(\tau)}=1$ implies $\|w^n\tilde p_{n,\alpha}\|_{L^2(\tau)}=\frac{1}{|a_{n,\alpha}|}$, which is minimal among polynomials $p\in \mathcal P(\alpha)$. 
		
		We need some control over these leading coefficients $a_{n,\alpha}$. For multiindices $\alpha \in \NN^d$ with $|\alpha|\leq n$, let
		\begin{equation}
			\label{Tna}
			T(n,\alpha):=\inf \{\|w^np\|_{L^2(\tau)}: p\in \mathcal P(\alpha)\}.
		\end{equation}
		Since $T(n,\alpha)= \frac{1}{|a_{n,\alpha}|}$ an equivalent issue is to control the numbers $T(n,\alpha)$.

		To proceed, following \cite{Bloom}, taking $K,w=e^{-Q}$ and $\tau$ as above, we let
		$$Z=Z(K,w):=\{(t,z)=(t,t\lambda)\in \CC \times \CC^d:\lambda\in K, \ |t|=w(\lambda)\}$$
		(note $\lambda = z/t\in \CC^d$) and
		$$Z(D):=\{(t,z)=(t,t\lambda):\lambda\in K, \ |t|\leq w(\lambda)\}.$$
		Then $Z(D)$ is a complete circular set ($\zeta Z(D) \subset Z(D)$ for $\zeta \in \CC$ with $|\zeta|\leq 1$). Since $V_{K,Q}$ is continuous, from Theorem 2.1 in \cite{Bloom}, $Z(D)$ is regular. Define a measure $\nu$ by the integration formula
		$$
		\int \phi d \nu = \int \left( \int \phi(t, t \lambda) d m_\lambda(t) \right) d \tau(\lambda)
		$$
		for measurable $\phi:\CC^{d + 1} \to \CC$. Here $m_{\lambda}$ is normalized arclength measure on the circle $|t|=w(\lambda)$ (see (3.3) of \cite{Bloom}). The probability measure $\nu$ satisfies the Bernstein-Markov property on $Z(D)$ in dimension $d + 1$ (see \cite[Theorem 3.1]{Bloom} and surrounding discussion). 
		
		 Given a polynomial $G_n(z)$ of degree at most $n$ in $\CC^d$, we can form the homogeneous polynomial
		$$P_n(t,z)=t^nG_n(z/t)$$
		of degree $n$ in $\CC^{d+1}$. Then 
		$$\|w^nG_n\|_K=\|P_n\|_{Z(D)} \quad \hbox{and} \quad \|w^nG_n\|_{L^2(\tau)}=\|P_n\|_{L^2(\nu)}.$$
		Next, applying Gram-Schmidt in $L^2(w^{2n}\tau)$ to the monomials $\{z^{\alpha}\}_{|\alpha|\leq n}$ in $\CC^d$ with the ordering $\prec$, we thus obtain orthonormal polynomials $\{p_{n,\alpha}\}_{|\alpha|\leq n}$ as in \eqref{pnal}. 
%
		Now, define
		$$P_{n,\alpha}(t,z):=t^np_{n,\alpha}(z/t).$$
		These are homogeneous polynomials of degree $n$ in $\CC^{d+1}$ which are orthonormal in $L^2(\nu)$.
		By extending the ordering $\prec$ to multi-indices in the variables $(t, z_1, \dots, z_d)$ so that $t$ comes before all the $z$ variables, i.e.,\
		$$
		(n-|\beta|,\beta)\prec (n-|\alpha|,\alpha) \ \hbox{if} \ \beta \prec \alpha,
		$$
		we have a similar expression to \eqref{pnal} for $P_{n, \alpha}$:
		$$P_{n,\alpha}(t,z)=a_{n,\alpha}t^{n-|\alpha|}z^{\alpha}+\sum_{\beta \prec \alpha} c_{\beta} t^{n-|\beta|}z^{\beta}.$$  
		Since
		$\|w^np_{n,\alpha}\|_{L^2(\tau)}=\|P_{n,\alpha}\|_{L^2(\nu)},$
		we similarly have
		$$T(n,\alpha)=\inf \{\|P\|_{L^2(\nu)}: P\in \mathcal P(n,\alpha)\}= \frac{1}{|a_{n,\alpha}|}$$
		where
		$$ \mathcal P(n,\alpha)=\mathcal P_{d+1}(n,\alpha):=\Big\{ t^{n-|\alpha|}z^{\alpha}+\sum_{\beta \prec \alpha} a_{\beta} t^{n-|\beta|}z^{\beta}: a_{\beta}\in \CC\Big\}$$
		are the ``monic'' homogeneous polynomials in $\CC^{d+1}$ corresponding to the class $\mathcal P(\alpha)$ in $\CC^d$. Thus, in order to control the leading coefficients $a_{n,\alpha}$ of $p_{n,\alpha}$ it suffices to work with the measure $\nu$ and the (same) leading coefficients $a_{n,\alpha}$ of $P_{n,\alpha}$. 
		
		Let $\Sigma_d:=\{(x_1,...,x_d)\in \RR^d: 0\leq x_i \leq 1, \ \sum_{j=1}^d x_i = 1\}$ be the standard simplex in $\RR^d$ and let 
		$$
		\Sigma_d^0:=\{(x_1,...,x_d)\in \RR^d: 0 < x_i <1, \ \sum_{j=1}^d x_i = 1\}.
		$$ Given a compact set $F\subset \CC^d$ and $\theta \in \Sigma_d$, Zaharjuta \cite{Z} defined directional Chebyshev constants $\tau(F,\theta)$ via 
		$$\tau(F,\theta):=\limsup_{\frac{\alpha}{|\alpha|}\to \theta}\tau_{\alpha}(F)$$
		where
		$$\tau_{\alpha}(F):=\inf \{\|p\|_F:p\in \mathcal P_d(\alpha)\}^{1/|\alpha|}.$$
		He showed that for $\theta\in \Sigma_d^0$, the limit
		$$\tau(F,\theta)= \lim_{\frac{\alpha}{|\alpha|}\to \theta}\tau_{\alpha}(F)$$
		exists and $\theta \to \tau(F,\theta)$ is continuous on $\Sigma_d^0$. For $\theta \in \Sigma \setminus \Sigma^0$, the limit need not exist (see Remark 1 in \cite{Z}). However, it was recently shown in \cite{BL} that for $F$ regular, the limit does, indeed, exist for $\theta \in \Sigma \setminus \Sigma^0$, and, moreover, in this case, $\theta \to \tau(F,\theta)$ is continuous on {\it all} of $\Sigma_d$. This fact will be crucial. 
		
		We write $\alpha(i)\in \NN^d$ for the $i$-th multi-index in our ordering $\prec$ on $\NN^d$. By \cite{BL} and the Bernstein-Markov property of $\nu$ on $Z(D)$, we know that for all $\tilde \theta=(\theta_0, \theta) \in \Sigma_{d+1}$, where $\theta_0\in [0,1]$, the limits in the directional Chebyshev constants for $Z(D)\subset \CC^{d+1}$ exist and are given by 
		\begin{equation}\label{lim} \tau(Z(D),\tilde \theta)=\lim_{(n-|\alpha|,\alpha)/n \to \tilde \theta } T(n, \alpha)^{1/n}=\lim_{(n-|\alpha|,\alpha)/n \to \tilde \theta }  \frac{1}{|a_{n,\alpha}|^{1/n}}.\end{equation}
		Here given $\alpha \in \NN^d$ so that $\alpha/|\alpha|=\theta \in \Sigma_d$, we have $(n-|\alpha|,\alpha)\in \NN^{d+1}$ and $(n-|\alpha|,\alpha)/n =\tilde \theta \in \Sigma_{d+1}$; and for $Z(D)$ circled ($(t,z)\in Z(D)$ implies $e^{i\theta}(t,z)\in Z(D)$) we only need consider these homogeneous Chebyshev constants $T(n,\alpha)$ to construct the directional Chebyshev constants for $Z(D)$ (cf., \cite{J}).
		
	We will need the following simple consequence of \eqref{lim}.

		\begin{proposition}\label{03} For any sequences $\{k_n\}, \{i_n\}\subset \NN$ with $k_n,i_n\leq m_n$ and 
			$$\|\alpha(i_n)-\alpha(k_n)\|_2=o(n),$$
			we have
			$$\lim_{n\to \infty} \frac{1}{n} \log \frac{|a_{n,\alpha(i_n)}|}{|a_{n,\alpha(k_n)}|}=0.$$
		\end{proposition}

		\begin{proof}
		Suppose not. Then we can find a subsequence $Y \subset \NN$ and $\epsilon  > 0$ such that
		$$
		\bigg|\frac{1}{n} \log \frac{|a_{n,\alpha(i_n)}|}{|a_{n,\alpha(k_n)}|} \bigg| > \epsilon
		$$
		for all $n \in Y$. Now, since $\Sigma_{d+1}$ is compact, we can find a further subsequence $Y' \subset Y$ such that $(n - |\alpha(i_n)|, \alpha(i_n))/n \in \Sigma_{d+1}$ converges to a point $\tilde \theta \in \Sigma_{d+1}$ along $Y'$. Since $\|\alpha(i_n)-\alpha(k_n)\|_2=o(n)$, we similarly have that $(n - |\alpha(k_n)|, \alpha(k_n))/n \to \tilde \theta$. Therefore applying \eqref{lim} we have that
		$$
		\lim_{n \in Y'} \frac{1}{n} \log |a_{n,\alpha(i_n)}| = \lim_{n \in Y'} \frac{1}{n} \log |a_{n,\alpha(k_n)}| = - \log \tau(Z(D), \tilde \theta),
		$$
		 which contradicts the previous display.
		\end{proof}

		\noindent We will only need this result for $i_n$ and $k_n$ where
		$$\alpha(i_n)=\alpha(k_n)\pm (0,...,0,1,0,...,0).$$

		\begin{remark} In the univariate case, given $j \in \{0,1,...,n\}$ corresponding to $\alpha(j)=j$, the point $(n-j,j)/n$ is a point on the simplex $\Sigma_2$. Using the notation in this setting, Proposition \ref{03} reads:
			
			\begin{proposition}\label{three} For any sequence $\{k_n\}$ with $k_n\in \NN$ and $k_n=o(n)$, and any sequence $\{i_n\}$ with $i_n\in \NN$ and $i_n \in \{1,...,n\}$, 
				$$\lim_{n\to \infty} \frac{1}{n}\log \bigl(\frac{|a_{n,i_n}|}{|a_{n,i_n+k_n}|}\bigr)=0.$$
				
			\end{proposition}
			
			\noindent The proof of Proposition \ref{three} only requires a special case of the $\CC^2$ version of the main result in \cite{BL}. We provide a self-contained proof of this case in the appendix.
		
		\end{remark}
		
		\begin{remark} In the case of a sequence $\{\tau_n\}$ of asymptotically Bernstein-Markov measures for $K,Q$, the corresponding sequence $\{\nu_n\}$ is asymptotically Bernstein-Markov on $Z(D)$. The  $T(n,\alpha)$ in (\ref{Tna}) are computed using $\tau_n$; from the asymptotic Bernstein-Markov property of $\{\nu_n\}$ and \cite{BL}, equation (\ref{lim}) is still valid. \end{remark}

		\section{Strategy of proof of Theorem \ref{mainthd}}
		
		The following multidimensional version of Theorem 4.1 of \cite{BD} or \cite{BDprint}, which occurs as Proposition 3.5 in \cite{BBL},  is the guiding light for the proof of Theorem \ref{mainthd}.
		
		\begin{proposition} \label{keypropB} Let $\{G_n(z,\omega)\}$ be a sequence of random polynomials of the form (\ref{probhb2}) such that for some continuous function $V\in L(\CC^d)$, 
			\begin{enumerate}
				\item almost surely, $\{\frac{1}{n}\log |G_n|\}$ is locally uniformly bounded above on $\CC^d$;  
				\item almost surely, $[\limsup_{n\to \infty}\frac{1}{n}\log |G_n(z,\omega)|]^*\leq V(z)$ for all $z \in \CC^d$; and 
				\item there is a countable dense set of points $\{w_r\}$ in $\CC^d$ such that 
				$$\lim_{n\to \infty}\frac{1}{n}\log |G_n(w_r,\omega)|=V(w_r), \ r=1,2,...$$
				almost surely. 
			\end{enumerate}
			Then almost surely, $\frac{1}{n}\log |G_n|\to V$ in $L^1_{loc}(\CC^d)$.		
		\end{proposition}
		
		To show that
		$$\frac{1}{n}\log |G_n|\to V_K \ \hbox{in} \ L^1_{loc}(\CC^d)$$ 
		in probability, it suffices to show that for any sequence $Y\subset \NN$ there is a further subsequence $Y_1\subset Y$ such that  
		$$\lim_{n\to \infty, \ n\in Y_1} \frac{1}{n}\log |G_n|= V_K \ \hbox{in} \ L^1_{loc}(\CC^d)$$
		almost surely. The first two properties in Proposition \ref{keypropB} of the full sequence $\{\frac{1}{n}\log |G_n|\}$ follow as in \cite{BBL} under the hypothesis that $\frac{1}{2n}\log B_n\to V_{K,Q}$ locally uniformly on $\CC^d$; this condition holds for bases satisfying (1) or (2) in the introduction by Proposition \ref{locunif}. Thus we concentrate on the third property for a subsequence.

		
From the previous observations, to apply Proposition \ref{keypropB} in this setting it suffices to prove the following.
		
		\begin{proposition} \label{01} Under the hypothesis \eqref{E:log13-multi}, there exists a countable dense set $\{z_j\}\subset \CC^d$ such that for any subsequence $Y\subset \NN$ there exists a further subsequence $Y_0\subset Y$ such that
			\begin{equation}\label{like3} \lim_{n\to \infty, \ n\in Y_0}\frac{1}{n}\log |G_n(z_j)|=V_{K,Q}(z_j) \end{equation}
			almost surely for every $z_j$.
		\end{proposition}
		
				
				Proposition \ref{01} follows from the next result.

				\begin{theorem} \label{02} Fix $\epsilon >0$, and $z\in \CC^d$. For $n\in \NN$, define
					$$J_n(z,\epsilon):=\# \{\alpha: |\alpha|\leq n \ \hbox{and} \ \frac{1}{n}\log|p_{n,\alpha}(z)|\geq V_{K,Q}(z)-\epsilon\}.$$
					Then there exists a set $M \subset \mathbb C^d$ of Lebesgue measure $0$ such that for any $z \in M^c$ and any $\epsilon > 0$, we have
					\begin{equation}
					\label{E:limnn}
						\lim_{n\to \infty} J_n(z,\epsilon) = \infty.
					\end{equation}
				\end{theorem}
				
				The proof of Theorem \ref{02} is the main novel idea in the present paper and is contained in the next section. Assuming Theorem \ref{02}, we can proceed similarly to \cite{BD} to prove Proposition \ref{01}. The key step deduces \eqref{like3} from the Kolmogorov-Rogozin inequality and \eqref{E:limnn}. This lemma explains how the condition \eqref{E:log13-multi} enters into the main theorem.
				
				\begin{lemma}
				\label{L:thm-41-variant}
				Let $\{z_{n, \alpha} : |\alpha| \le n, \alpha \in \NN^d, n \in \NN \}$ be an array of complex numbers and for $M \in \RR$  define
				$$
				K_n(M) = \# \{\alpha : |\alpha| \le n \text{ and } \tfrac{1}{n} \log |z_{n, \alpha}| \ge M\}.
				$$
				Suppose there exists $m_0 \in \RR$ such that 
				\begin{equation}
					\label{E:twin-conditions}
\lim_{n \to \infty} K_n(M) = \infty, \; M < m_0, \qquad \lim_{n \to \infty} K_n(M) = 0, \; M > m_0.
				\end{equation}
				
				Let $\{\xi_{n, \alpha} : |\alpha| \le n, \alpha \in \NN^d, n \in \NN \}$ be an array of i.i.d.\ non-degenerate complex random variables satisfying \eqref{E:log13-multi}. Then
				$$
				\lim_{n \to \infty} \frac{1}{n} \log \big| \sum_{|\alpha| \le n} \xi_{n, \alpha} z_{n, \alpha} \big| = m_0
				$$ 
				in probability.
				\end{lemma}
				
				The proof is a slight variant of the first part of the proof of Theorem 5.2 in \cite{BD}.
				
				\begin{proof}
					We first show that for any $\epsilon > 0$, that
					\begin{equation}
					\label{E:upper-bd}
					\lim_{n \to \infty} \PP \big(\frac{1}{n} \log \big| \sum_{|\alpha| \le n} \xi_{n, \alpha} z_{n, \alpha} \big| > m_0 + \epsilon \big) = 0.
					\end{equation}
					We have the bound
					$$
					\frac{1}{n} \log \big| \sum_{|\alpha| \le n} \xi_{n, \alpha} z_{n, \alpha} \big| \le \frac{1}{n} \log \left(\max_{|\alpha| \le n} |\xi_{n, \alpha} |\right) + \frac{1}{n} \log \left(\max_{|\alpha| \le n} |z_{n, \alpha} |\right) + \frac{d \log n}{n}.
					$$
					The second term on the right-hand side above is at most $m_0 + \epsilon/2$ for large enough $n$ since $K_n(m_0 + \epsilon/2) = 0$ for large enough $n$, and the third term tends to $0$ with $n$. To bound the first term, observe that by a union bound and condition \eqref{E:log13-multi},
					\begin{align*}
	\PP(\log |\xi_{n, \alpha} | &> n\epsilon/4 \text{ for some } |\alpha| \le n) \\
	&\le \binom{n+d}{d} \PP(\log |\xi_{n, 0} | > n\epsilon/4) = o(1), 
					\end{align*}
					which implies that the first term in the previous display is at most $\epsilon/4$ with probability tending to $1$ with $n$. This yields \eqref{E:upper-bd}.
				
				To complete the proof, we show that for any $\epsilon > 0$, 
				\begin{equation}
					\label{E:lower-bd}
					\lim_{n \to \infty} \PP \big(\frac{1}{n} \log \big| \sum_{|\alpha| \le n} \xi_{n, \alpha} z_{n, \alpha} \big| < m_0 - \epsilon \big) = 0.
				\end{equation}
				To prove \eqref{E:lower-bd}, recall that for a complex random variable $X$ and $r > 0$, its concentration function is
				$$
				\mathcal Q(X; r) := \sup_{c \in \CC} \PP(|X - c| < r).
				$$
				Using the concentration, we can bound the probability in \eqref{E:lower-bd} above by
				$$
				\mathcal Q\Big( \sum_{|\alpha| \le n} \xi_{n, \alpha} z_{n, \alpha} \;;\; e^{n(m_0 - \epsilon)} \Big), 
				$$
				which by the Kolmogorov-Rogozin inequality for sums of independent random variables (see \cite{E}, Corollary 1) is bounded above by
				\begin{equation}
					\label{E:CBIG}
	C \Big(\sum_{|\alpha| \le n} [1-\mathcal Q( \xi_{n, \alpha} z_{n, \alpha}  \;;\; e^{n(m_0 - \epsilon)} )] \Big)^{-1/2},
				\end{equation}
				 where $C> 0$ is an absolute constant. Now, if $|z_{n, \alpha}| > e^{n(m_0 - \epsilon/2)}$, then
				 $$
				 \mathcal Q( \xi_{n, \alpha} z_{n, \alpha}  \;;\; e^{n(m_0 - \epsilon)} ) \le \mathcal Q( \xi_{n, \alpha} e^{n(m_0 - \epsilon/2)} \;;\; e^{n(m_0 - \epsilon)} ) = \mathcal Q(\xi_{n, \alpha}; e^{-n \epsilon/2}).
				 $$
				Since the $\xi_{n, \alpha}$ are non-degenerate, for large enough $n$, the right-hand side is bounded above by $1 - \delta$ for some $\delta > 0$. Therefore for such $n$, \eqref{E:CBIG} is bounded above by
				$$
				C [K_n(m_0 - \epsilon/2) \delta ]^{-1/2},
				$$
				which tends to $0$ with $n$ since $K_n(m_0 - \epsilon/2) \to \infty$. This proves \eqref{E:lower-bd}, completing the proof of the lemma.
				\end{proof}

				\begin{proof}[Proof of Proposition \ref{01}] It suffices to find a dense set $M^c \subset \CC^d$ such that for any $z \in M^c$ we have
					\begin{equation}
					\label{E:like33}
					\lim_{n\to \infty}\frac{1}{n}\log |G_n(z)|=V_{K,Q}(z) \qquad \text{ in probability.}
					\end{equation}
					Then, given a countable dense subset $\{z_j\} \subset M^c$, given any subsequence $Y$ we can use a diagonalization argument to pass to a further subsequence $Y_0$ on which \eqref{E:like33} holds almost surely, as required.
					
				Take $M$ as in Theorem \ref{02}, so that for $z \in M^c$ we have
				\begin{equation}\label{key3} \lim_{n\to \infty}  J_n(z,\epsilon)=\infty
				\end{equation}
				for all $\epsilon > 0$. We apply Lemma \ref{L:thm-41-variant} with $z_{n, \alpha} = p_{n, \alpha}(z)$. With notation as in that lemma, \eqref{key3} implies that $K_n(V_{K, Q}(z) - \epsilon) \to \infty$ for all $\epsilon > 0$. Moreover for any $\epsilon > 0$, $K_n(V_{K, Q}(z) + \epsilon) = 0$ for all large enough $n$ by Proposition \ref{locunif}. Applying Lemma \ref{L:thm-41-variant} yields \eqref{E:like33}.
				\end{proof}
				
				\begin{proof}[Proof of Theorem \ref{mainthd}, Theorem \ref{T:mainth-dim1}, Corollary \ref{C:zero-currents}]
				By Proposition \ref{01} and the discussion after Proposition \ref{keypropB}, we can find a countable dense set $\{z_j\} \subset \CC^d$ such that given any subsequence $Y \subset \NN$ we can find a further subsequence $Y' \subset Y$ such that conditions (1)-(3) in Proposition \ref{keypropB} hold almost surely along $Y'$ with $V = V_{K, Q}$, and hence $\tfrac{1}{n} \log |G_n| \to V_{K, Q}$ in $L^1_{loc}(\CC^d)$ almost surely along this subsequence. 	Hence $\tfrac{1}{n} \log |G_n| \to V_{K, Q}$ in probability over $\NN$, proving Theorem \ref{mainthd}. Moreover, the first bullet point in Corollary \ref{C:zero-currents} then follows from the weak*-continuity of the operator $dd^c$ in $L^1_{loc}(\CC^d)$. 
				
				For the second bullet in Corollary \ref{C:zero-currents}, letting $O$ be any weak*-open set of $(1, 1)$-currents, containing $dd^c V_{K, Q}$, continuity of $dd^c$ implies that $(dd^c)^{-1} O$ is an open set in $L^1_{loc}(\CC^d)$ containing $V_{K, Q}$. Applying Theorem \ref{mainthd} gives that 
				$$
				\lim_{n \to \infty} \PP(\tfrac{1}{n} \log |G_n| \in (dd^c)^{-1} O) = 1,
				$$
				yielding the second bullet. Finally, Theorem \ref{T:mainth-dim1} is the special case when our $(1, 1)$-currents can be identified with measures with the (now metrizable) weak* topology. Metrizability ensures that we can recast either of the equivalent bullet points in Corollary \ref{C:zero-currents} as convergence in probability.
				\end{proof}
				
				\section{Proof of Theorem \ref{02}}
				
					We will give the proof of Theorem \ref{02} under condition (1); i.e., where $K\subset \CC^d$ is compact and locally regular; $w=e^{-Q}$ is continuous on $K$; and $\tau$ is a finite measure on $K$ such that $(K,Q,\tau)$ satisfies a weighted Bernstein-Markov property (\ref{wbwin}). The proof in the case where we take a sequence $\{\mu_n\}$ of asymptotically weighted Bernstein-Markov measures for $K,Q$ is the same, mutatis mutandis.
				
				Theorem \ref{02} naturally splits into two cases, depending on the value of $V_{K, Q}$ at $z$. If $V_{K, Q}(z)$ attains its minimal value $\inf Q$, then the polynomials for which $\frac{1}{n}\log|p_{n,\alpha}(z)|$ is close to $V_{K, Q}(z)$ have low degree. In this case, we can show that $J_n(z, \epsilon)$ is large by directly showing that when $|\alpha| \ll n$, the function
				$\frac{1}{n}\log|p_{n,\alpha}(z)|$ is essentially flat and close to $\inf Q$ outside of a set of small measure.
				
				On the other hand, at points where $V_{K, Q}(z) > \inf Q$, the polynomials used in the approximation of $V_{K, Q}(z)$ will have large degree. Large degree orthonormal polynomials for arbitrary $K, Q$ are harder to control explicitly; our construction simply shows the existence of a collection of polynomials that ensure $J_n(z, \epsilon)$ is large. The rough idea is to use the orthogonality condition, as well as Proposition \ref{03}, to bootstrap the existence of a single $\alpha, \ |\alpha|\leq n$ with $\frac{1}{n}\log|p_{n,\alpha}(z)|$ close to $V_{K, Q}(z)$ -- guaranteed by Proposition \ref{locunif} -- to a collection of other $\beta, \ |\beta|\leq n$ with $\frac{1}{n}\log|p_{n,\beta}(z)|$ close to $V_{K, Q}(z)$.
				
				To understand why we need to introduce two cases, it may be useful to keep in mind the simplest case of the sequence of single-variable monomials $1, z, z^2, \dots$, which are the orthonormal polynomials for normalized arclength measure on the unit circle $K$ in the unweighted case $Q = 0$. 
				When $0 < |z| \le 1$, we have
				\begin{equation}
					\label{E:kac-example}
					\frac{1}{n} \log |z^{k_n}| \to V_{K, Q}(z) = \max( \log |z|, 0)
				\end{equation}
				for any sequence $k_n = o(n)$, whereas when $|z| \ge 1$, the convergence in \eqref{E:kac-example} holds whenever $k_n = n - o(n)$. The point $z = 0$ is the unique point where $J_n(z, \epsilon) \not \to \infty$ as $n \to \infty$ for some $\epsilon > 0$; this is the set $M$ in Theorem \ref{02}.
				
				We first deal with the case where $V_{K, Q}(z) = \inf Q$. 
				For this, we require a multivariate version of Cartan's estimate on the Lebesgue measure of the set where a polynomial can be small. We will use this estimate to show that for low degree polynomials, $\frac{1}n \log |p_{n, \alpha}(z)|$ is essentially flat.  The following lemma is Corollary 4.2 of \cite{BJZ}.
				
				\begin{lemma} \label{cartan2} Let $B$ be a closed ball. There exists $C=C(d)$ such that for any polynomial $p$ of degree $n$ with $\|p\|_B=1$ we have
					$$\frac{\mathcal L(\{z\in B:|p(z)|\leq \epsilon^n\})}{\mathcal L(B)}\leq C\epsilon^2$$
					for $0<\epsilon <1$. Here $\mathcal L$ denotes Lebesgue measure.
				\end{lemma}
				
				\begin{remark} In the univariate case, we have the stronger classical Cartan's lemma: {\it for a monic polynomial $p$ of degree $n$ and any $h>0$,} 
					$$\mathcal L (\{z: |p(z)|\leq h^n\})\leq ch^2.$$
					
				\end{remark}
				
				\begin{lemma}
					\label{L:poly-small}
					Let $\alpha_n, n \in \mathbb N$ be any sequence with $|\alpha_n| = o(n/\log n)$. Then
					$$
					\lim_{n \to \infty} \frac{1}n \log |p_{n, \alpha_n}(z)| = \inf Q
					$$
					for (Lebesgue) almost every $z \in \mathbb C^d$ and
					$$
					\limsup_{n \to \infty} \frac{1}n \log |p_{n, \alpha_n}(z)| \le \inf Q
					$$
					for every $z \in \mathbb C^d$.
				\end{lemma}
				
				\begin{proof}
					Fix a closed ball $B \subset \mathbb C^d$. We first claim that
					\begin{equation}
						\label{E:log-comp}
						\lim_{n \to \infty} \frac{1}n \log \frac{|p_{n, \alpha_n}(z)|}{\|p_{n, \alpha_n}\|_B} = 0 
					\end{equation}
					for almost every $z \in B$. Indeed, for any fixed $\epsilon > 0$, by Lemma \ref{cartan2}, we have
					$$
					\mathcal L \left\{z \in B : \frac{1}n \log \frac{|p_{n, \alpha_n}(z)|}{\|p_{n, \alpha_n}\|_B} \le -\epsilon \right\} \le \mathcal L(B)\cdot  Ce^{-2\epsilon n/|\alpha_n|}
					$$
					where the constant $C$ is as in this lemma. Since $|\alpha_n| = o(n/\log n)$, the right-hand side is summable in $n$ for all fixed $\epsilon > 0$. Therefore by the Borel-Cantelli lemma, 
					$$
					\liminf_{n \to \infty} \frac{1}n \log \frac{|p_{n, \alpha_n}(z)|}{\|p_{n, \alpha_n}\|_B} \ge - \epsilon
					$$
					for almost every $z \in B$. Since $\epsilon > 0$ was arbitrary, \eqref{E:log-comp} follows. 
					
					Now, for all $z \in \mathbb C^d$, from the definition of $V_{K, Q}$ we have
					$$
					\frac{1}{n} \log \frac{|p_{n, \alpha_n}(z)|}{\|p_{n, \alpha_n} e^{-nQ} \|_K} \le V_{K, Q}(z).
					$$
					Noting that 
					$\tfrac{1}n \log \|p_{n, \alpha_n} w^n \|_K \to 0$
					as $n \to \infty$ by the weighted Bernstein Markov property
					(\ref{wbwin}), we therefore have
					\begin{equation}
						\label{E:limsupp}
						\limsup_{n \to \infty} \frac{1}{n} \log |p_{n, \alpha_n}(z)| \le V_{K, Q}(z)
					\end{equation}
					for every $z \in \mathbb C^d$. On the other hand, letting $B$ be any closed ball containing $K$, we have
					$$
					\|p_{n, \alpha_n}\|_B \ge \|p_{n, \alpha_n}\|_K \ge e^{n \inf Q} \|p_{n, \alpha_n} e^{-nQ}\|_K = e^{n (\inf Q + o(1))},
					$$
					where the final bound again uses the weighted Bernstein-Markov property (\ref{wbwin}), thus
					\begin{equation}
						\label{E:liminff}
						\liminf_{n \to \infty} \frac{1}{n} \log \|p_{n, \alpha_n}\|_B \ge \inf Q.
					\end{equation}
					Now, since $Q$ is continuous and $V_{K, Q} \le Q$ on $K$, the only way \eqref{E:log-comp}, \eqref{E:limsupp}, and \eqref{E:liminff} can simultaneously hold is if 
					$$
					\lim_{n \to \infty} \frac{1}{n} \log |p_{n, \alpha_n}(z)| = \inf Q = \lim_{n \to \infty} \frac{1}{n} \log \|p_{n, \alpha_n}\|_B 
					$$ 
					for almost every $z \in B$. This first equality gives the first claim in the lemma, and the second equality gives the second claim.
				\end{proof}
				
				We can now quickly prove the $V_{K, Q}(z) = \inf Q$ case of Theorem \ref{02}.
				
				\begin{corollary}
					\label{C:theorem33-special}
					There exists a set $M \subset \mathbb C^d$ of Lebesgue measure $0$ such that for any $z \in M^c$ with $V_{K, Q}(z) = \inf Q$ and any $\epsilon > 0$, we have
					$$ 
					\lim_{n\to \infty} J_n(z,\epsilon) = \infty.
					$$
				\end{corollary} 
				
				\begin{proof}
					Let $\alpha(0), \alpha(1), \dots$ be the sequence of multi-indices in $\mathbb N^d$, listed according to the total order $\prec$. For every $k$, let $M_k$ be the set of points $z \in \mathbb C^d$ where
					$$
					\lim_{n \to \infty} \frac{1}n \log |p_{n, \alpha(k)}(z)| \ne \inf Q,
					$$
					and set $M = \bigcup_{k=1}^\infty M_k$. The Lebesgue measure of each $M_k$ is $0$ by Lemma \ref{L:poly-small}, and so the Lebesgue measure of $M$ is also $0$. Moreover, the set $M$ satisfies the other condition of the corollary by construction.
				\end{proof}
				
				The remaining case of Theorem \ref{02} follows from the next lemma.
				
				\begin{lemma}
					\label{L:general-case}
					For any $z \in \mathbb C$ with $V_{K, Q}(z) > \inf Q$ and any $\epsilon > 0$, we have
					$$ 
					\lim_{n\to \infty} J_n(z,\epsilon) = \infty.
					$$
				\end{lemma}
				
				\begin{proof}[Proof of Theorem \ref{02} given Lemma \ref{L:general-case}]
					The constant function $\inf Q$ belongs to $L(\mathbb C^d)$; thus $V_{K, Q} \ge \inf Q$ in $\CC^d$.
					Taking $M$ as in Corollary \ref{C:theorem33-special}, the theorem follows by combining Corollary \ref{C:theorem33-special} with Lemma \ref{L:general-case}.
				\end{proof}
				
				\begin{proof}[Proof of Lemma \ref{L:general-case}]
					Fix $z \in \mathbb C$ with $V_{K, Q}(z) > \inf Q$. To prove the lemma, it is enough to show that for all $k \in \mathbb N$,
					\begin{equation}
						\label{E:epsilon-0}
						\liminf_{n \to \infty} J_n(z, \epsilon) \ge k \qquad \text{ for all } \epsilon > 0.
					\end{equation}
					We prove \eqref{E:epsilon-0} by induction on $k$. For the base case when $k = 1$, observe that
					$$
					\max_{\alpha : |\alpha| \le n} \frac{1}n \log |p_{n, \alpha}(z)| \ge \frac{1}{2n} \log \sum_{\alpha: |\alpha| \le n} |p_{n, \alpha}(z)|^2 - \frac{1}{2n} \log \binom{n+d}{d}.
					$$
					The first term on the right hand side converges to $V_{K, Q}(z)$ as $n \to \infty$ by Proposition \ref{locunif} and the second term on the right-hand side above converges to $0$. Moreover, using the same argument as in (\ref{E:limsupp})
					$$V_{K,Q}(z)\geq \limsup_{n\to \infty} \bigl(\max_{\alpha : |\alpha| \le n} \frac{1}n \log |p_{n, \alpha}(z)|\bigr),$$
					and so
					$$ \lim_{n\to \infty}\bigl( \max_{\alpha : |\alpha| \le n} \frac{1}n \log |p_{n, \alpha}(z)|\bigr)=V_{K,Q}(z).$$ 
					Thus for any fixed $\epsilon > 0$, $J_n(z, \epsilon) \ge 1$ for all large enough $n$.
					
					Now suppose that \eqref{E:epsilon-0} holds for $k \ge 1$. We prove it for $k + 1$. Fix $\epsilon > 0$ small enough so that 
					$V_{K, Q}(z) - \epsilon > \inf Q$. By the inductive hypothesis, for all large enough $n$ we can find indices $\alpha_i = \alpha_i(z, n)$ with $|\alpha_i|\leq n$ and 
					$$
					\alpha_k \prec \alpha_{k-1}\prec \dots \prec \alpha_1 
					$$
					such that for all $i \in \{1, \dots, k\}$ we have
					\begin{equation}
						\label{eightbis}
						\frac{1}n \log |p_{n, \alpha_i}(z)| \ge V_{K, Q}(z) - \epsilon/2 > \inf Q.
					\end{equation}
					As long as $n$ was chosen large enough, Lemma \ref{L:poly-small} and the second inequality in \eqref{eightbis} ensure that $\alpha_k=(\alpha_{k,1},...,\alpha_{k,d})  \ne (0, \dots, 0)$. Therefore we can find $j \in \{1, \dots, d\}$ such that $\alpha_{k, j} \ge 1$. Let $e_j = (0, \dots, 1, \dots, 0)$ be the $j$th coordinate vector, set $\alpha_k^- =\alpha_k^-(z,n):= \alpha_k - e_j$, and define a polynomial
					$$
					q_{n,\alpha_k}(z):=z_jp_{n,\alpha_k^-}(z).
					$$
					We write $\beta \preceq \alpha$ if $\beta \prec \alpha$ or $\beta = \alpha$. From our hypothesis on $\prec$ that $\alpha \prec \beta \implies \alpha + e_j \prec \beta + e_j$, we can expand $q_{n,\alpha_k}$ as a linear combination of $p_{n,\beta}$ with $\beta \preceq \alpha_k$:
					$$
					q_{n,\alpha_k}:=\sum_{\beta \preceq \alpha_k} \lambda_\beta p_{n,\beta}.
					$$
					Then
					$$
					\sqrt {\sum_{\beta \preceq \alpha_k} |\lambda_{\beta}|^2} = \|q_{n,\alpha_k}\|_{L^2(w^{2n}\tau)}\leq[ \max_{z\in K}|z_j|] \cdot \|p_{n,\alpha_k^-}\|_{L^2(w^{2n}\tau)}=\max_{z\in K}|z_j|
					$$
					and so each $|\lambda_\beta|$ is bounded above by the constant $c_K := \max_{z\in K}|z_j|$. Next, 
					$$
					\lambda_{\alpha_k} p_{n,\alpha_k}=q_{n, \alpha_k}-\sum_{\beta \prec \alpha_k} \lambda_\beta p_{n,\beta}
					$$
					and so applying the triangle inequality at $z$ we have
					\begin{align*}
						|\lambda_{\alpha_k} p_{n,\alpha_k}(z)|&\leq |q_{n,\alpha_k}(z)|+ \sum_{\beta \prec \alpha_k} |\lambda_\beta| \cdot |p_{n,\beta}(z)| \\
						&\leq (|z_j|+|\lambda_{\alpha_k^-}|)|p_{n,\alpha_k^-}(z)|+\sum_{\beta \prec \alpha_k, \beta \ne \alpha_k^-} |\lambda_{\beta} | \cdot |p_{n,\beta}(z)|.
					\end{align*}
					Therefore by the above inequality, letting $c_{K, z} := c_K + |z_j|$ we have
					$$
					|\lambda_{\alpha_k} p_{n,\alpha_k}(z)|\leq c_{K, z}\sum_{\beta \prec \alpha_k} |p_{n,\beta}(z)|.
					$$
					Next, observe that the leading coefficient of $q_{n,\alpha_k}$, i.e., the coefficient of $z^{\alpha_k}$, is equal to $\lambda_{\alpha_k}$ times the leading coefficient $a_{n,\alpha_k}$ of $p_{n,\alpha_k}$. Moreover, it is also the leading coefficient $a_{n,\alpha_k^-}$ of $z_jp_{n,\alpha_k^-}$. Thus 
					$$
					\lambda_{\alpha_k}= \frac{a_{n,\alpha_k^-}}{a_{n,\alpha_k}}
					$$
					and our estimate above on $|\lambda_{\alpha_k} p_{n,\alpha_k}(z)|$ becomes
					\begin{equation}
						\label{10bis}
						\begin{split}
							\frac{|a_{n,\alpha_k^-}|}{|a_{n,\alpha_k}|}|p_{n,\alpha_k}(z)|&\leq c_{K, z}\sum_{\beta \prec \alpha_k} |p_{n,\beta}(z)| \\
							&\le c_{K, z} \binom{n + d}{d} \max_{\beta \prec \alpha_k} |p_{n,\beta}(z)|.
						\end{split} 
					\end{equation}
					Taking logs and then combining \eqref{10bis} with \eqref{eightbis} gives
					$$
					\max_{\beta \prec \alpha_k} \frac{1}n \log |p_{n,\beta}(z)| \ge V_{K, Q}(z) - \epsilon/2 - \frac{1}n \log \binom{n+d}{n} + \frac{1}n \log \frac{|a_{n,\alpha_k^-(z, n)}|}{|a_{n,\alpha_k(z, n)}|}. 
					$$
					By Proposition \ref{03}, the final term above tends to $0$ as $n \to \infty$. Therefore for all large enough $n$, there exists $\alpha_{k+1} \prec \alpha_k$ such that 
					$$
					\frac{1}n \log |p_{n,\alpha_{k+1}}(z)| \ge V_{K, Q}(z) - \epsilon,
					$$
					and so $J_n(z, \epsilon) \ge k+1$. Since $\epsilon > 0$ was arbitrary, this completes the inductive step in \eqref{E:epsilon-0}.
				\end{proof}
				
				\section{Appendix}
				
				  We give a self-contained proof of a special $\CC^2$ case of the main result in \cite{BL} which is all that is needed for the proof of Proposition \ref{three} and hence the univariate version of Theorem \ref{mainthd} and hence Theorem \ref{T:mainth-dim1}. 
				
				\begin{theorem} \label{maint} For $K\subset \CC^2$ compact, regular, and complete circular, if $\theta \in \Sigma_2 \setminus \Sigma_2^0=\{(0,1),(1,0)\}$, 
					$$\tau(K,\theta)= \lim_{\frac{\alpha}{|\alpha|}\to \theta}\tau_{\alpha}(K);$$ i.e., the limit exists. Moreover, $\theta \to \tau(K,\theta)$ is continuous on $\Sigma_2$.
				\end{theorem}
				
				We begin with some general facts from \cite{Z} for an arbitrary compact set $K\subset \CC^d$. From \cite{Z}, it is known that for $\theta_d \in \Sigma_d^0$, the limit
				$$\tau(K,\theta):= \lim_{\frac{\alpha}{|\alpha|}\to \theta}\tau_{\alpha}(K)$$
				exists and $\theta \to \tau(K,\theta)$ is continuous on $\Sigma_d^0$ where we recall that $\tau_{\alpha}(K):=\inf \{\|p\|_K:p\in \mathcal P_d(\alpha)\}^{1/|\alpha|}$. Moreover, recall that for circled (and hence complete circular) sets we only need consider homogeneous polynomials in $\mathcal P_d(\alpha)$. For $\theta \in \Sigma_d \setminus \Sigma_d^0$, we define
				\begin{equation} \label{bdrydirl} \tau(K,\theta):= \limsup_{\frac{\alpha}{|\alpha|}\to \theta}\tau_{\alpha}(K) \end{equation}
				and  Lemma 3 in \cite{Z} shows that for $\theta \in \Sigma_d \setminus \Sigma_d^0$ one has  
				\begin{equation}\label{liminf} \tau_{-}(K,\theta):=\liminf_{\frac{\alpha}{|\alpha|}\to \theta}\tau_{\alpha}(K)= \liminf_{\theta'\to \theta, \ \theta'\in \Sigma_d^0}\tau(K,\theta').\end{equation}
				
			\begin{proof}[Proof of Theorem \ref{maint}]
				Let $K\subset \CC^2$ be regular. We write $\Sigma:=\Sigma_2$ and $\Sigma^0:=\Sigma_2^0$. Here we order the monomials as usual: 
				$$1,z,w,...,z^n,z^{n-1}w,...,w^n,... $$ and we write $\alpha =(a,b)$. We first take $(0,1)\in \Sigma \setminus \Sigma^0$ and we want to show
				$$\lim_{{\alpha \over |\alpha|}\to (0,1)}\tau_{\alpha}(K)$$
				exists. The first observation is that 
					\begin{equation} \label{biumj} \lim_{n\to \infty}\tau_{(0,n)}(K) \ \hbox{exists and equals} \  \liminf_{\theta \to (0,1), \ \theta \in \Sigma^0}\tau(K,\theta).\end{equation}
					This was shown by the first author in \cite{B-IUMJ} (see Proposition 2.1 and Theorem 3.2). This uses an ``internal ellipsoid'' argument which we will also use below in the case $(1,0) \in \Sigma \setminus \Sigma^0$.

				Thus we need to consider all sequences with $\theta_j:=\frac{\alpha_j}{|\alpha_j|}\to (0,1)$ and show 
				$\lim_{j\to \infty}\tau_{\alpha_j}(K)$ exists and equals $\lim_{n\to \infty}\tau_{(0,n)}(K)$. 
				Given such a sequence, we can write
				$$ \theta_j=(\epsilon_j,1-\epsilon_j) \ \hbox{and} \ \alpha_j=n_j\theta_j$$
				where $n_j=|\alpha_j|$. Note that $n_j \epsilon_j$ is an integer. Let 
				$$p_{0,n_j(1-\epsilon_j)}(z,w):=w^{n_j(1-\epsilon_j)}\pm \cdots$$
				be the Chebyshev polynomial among polynomials of this type; i.e., $\|p_{0,n_j(1-\epsilon_j)}\|_K^{1/n_j(1-\epsilon_j)}= \tau_{(0,n_j(1-\epsilon_j))}(K)$. Then 
				$$q_{n_j\epsilon_j,n_j(1-\epsilon_j)}(z,w):=z^{n_j\epsilon_j}\cdot p_{0,n_j(1-\epsilon_j)}(z,w)= w^{n_j(1-\epsilon_j)}z^{n_j\epsilon_j}  + \cdots$$
				is a competitor for the Chebyshev polynomial in this class, i.e., 
				$$\|q_{n_j\epsilon_j,n_j(1-\epsilon_j)}\|_K^{1/n_j} \geq  \tau_{(n_j\epsilon_j,n_j(1-\epsilon_j))}(K).$$
				Thus
				$$\tau_{(n_j\epsilon_j,n_j(1-\epsilon_j))}(K)\leq \|q_{n_j\epsilon_j,n_j(1-\epsilon_j)}\|_K^{1/n_j} \leq \tau_{(0,n_j(1-\epsilon_j))}(K)^{1-\epsilon_j}\cdot \max_{(z,w)\in K} |z|^{\epsilon_j}.$$
				Since we can assume $\max_{(z,w)\in K} |z|>0$ ($K$ is not pluripolar),
				$$ \limsup_{j\to \infty}  \tau_{(n_j\epsilon_j,n_j(1-\epsilon_j))}(K) \leq \limsup_{n_j\to \infty}\tau_{(0,n_j(1-\epsilon_j))}(K)=\limsup_{n\to \infty}\tau_{(0,n)}(K).$$   
				On the other hand, using (\ref{biumj}) and (\ref{liminf}),
				$$\liminf_{j\to \infty}\tau_{(n_j\epsilon_j,n_j(1-\epsilon_j))}(K)\geq \liminf_{\frac{\alpha}{|\alpha|}\to (0,1)}\tau_{\alpha}(K)$$
				$$=\liminf_{\theta \to (0,1), \ \theta \in \Sigma^0}\tau(K,\theta)=\lim_{n\to \infty}\tau_{(0,n)}(K)$$
				so that, in fact, 
				$$\lim_{j\to \infty}\tau_{(n_j\epsilon_j,n_j(1-\epsilon_j))}(K) \ \hbox{exists and equals} \ \lim_{n\to \infty}\tau_{(0,n)}(K)$$
				for any sequence $\{\alpha_j=n_j\theta_j =(n_j\epsilon_j,n_j(1-\epsilon_j))\}$ with $|\alpha_j|\to \infty$ and $\theta_j:=\frac{\alpha_j}{|\alpha_j|}\to (0,1)$.
				
				The above argument did not require $K$ to be complete circular. But to get the analogous result at the other boundary point, $(1,0)$, we will assume that $K$ is complete circular; and since $||p||_K=||p||_{\hat K}$ for any polynomial $p$ we may assume $K=\hat K$ (recall (\ref{phull})). Thus the intersection of $K$ with any complex line through the origin is a closed disk centered at the origin in this line, and each disk has positive radius since $K$ is not pluripolar.
				
				First, to compute $\tau_{(n,0)}(K)$, we minimize sup norms on $K$ over the {\it single} homogeneous polynomial $z^n$: thus $\tau_{(n,0)}(K)=\max_{(z,w)\in K} |z|$ and 
				$$\lim_{n\to \infty}\tau_{(n,0)}(K)=\max_{(z,w)\in K} |z|$$
				(which is positive since $K$ is not pluripolar). Take any sequence $\alpha_j$ with $|\alpha_j|\to \infty$ and $\theta_j:=\frac{\alpha_j}{|\alpha_j|}\to (1,0)$. Clearly we can write
				$$\theta_j=(1-\epsilon_j,\epsilon_j) \ \hbox{and} \ \alpha_j=n_j\theta_j$$
				(where $n_j=|\alpha_j|$). We consider $\tau_{(n_j(1-\epsilon_j),n_j\epsilon_j)}(K)$ where we minimize sup norms on $K$ over the homogeneous polynomials of the form
				$$z^{n_j(1-\epsilon_j)}w^{n_j\epsilon_j}+...+ a_nz^n=z^{n_j(1-\epsilon_j)}(w^{n_j\epsilon_j}+...+a_nz^{n_j\epsilon_j});$$
				i.e., the competitors for $\tau_{(n_j(1-\epsilon_j),n_j\epsilon_j)}(K)$ are precisely $z^{n_j(1-\epsilon_j)}$ times the competitors for $\tau_{(0,n_j\epsilon_j)}(K)$. Hence
				$$\tau_{(n_j(1-\epsilon_j),n_j\epsilon_j)}(K)^{n_j}\leq \tau_{(0,n_j\epsilon_j)}(K)^{n_j\epsilon_j}\cdot [\max_{(z,w)\in K} |z|^{n_j(1-\epsilon_j)}]$$
				showing that 
				$$\limsup_{j\to \infty}\tau_{(n_j(1-\epsilon_j),n_j\epsilon_j)}(K)\leq \limsup_{j\to \infty} \max_{(z,w)\in K} |z|^{1-\epsilon_j}\cdot \tau_{(0,n_j\epsilon_j)}(K)^{\epsilon_j}$$
				$$=\max_{(z,w)\in K} |z|=\lim_{n\to \infty}\tau_{(n,0)}(K).$$
				Thus we have shown that 
				\begin{equation}\label{Rlim} \limsup_{\frac{\alpha}{|\alpha|}\to (1,0)}\tau_{\alpha}(K)\leq \max_{(z,w)\in K} |z|=\lim_{n\to \infty}\tau_{(n,0)}(K).\end{equation}

				For a reverse inequality (with $\liminf$), we use the hypothesis that $K$ is complete circular. Suppose first that for $K_0:=K\cap \{w=0\}=\{(z,0)\in K\}$ we have
				$$\max\{|z|: (z,0) \in K_0\}=\max_{(z,w)\in K}|z|=:R.$$
				Recall $\lim_{n\to \infty}\tau_{(n,0)}(K)=\max_{(z,w)\in K}|z|=R$. We can use interior approximation by an ellipsoid to get the lower bound for the $(1,0)$ Chebshev constant of $K$. To be precise, for $r<R$, the closed disk 
				$$\{(z,0):|z|\leq r\}$$
				is contained in the interior of $K$ and hence for some constant $A=A(r)$ the ellipsoid
				$$E=E(A,r):=\{(z,w): |z|^2/r^2 + |w|^2/A^2 \leq 1\}$$
				is contained in the interior of $K$ (cf., the proof of Theorem 3.2 in \cite{B-IUMJ}). Thus for all $\theta \in \Sigma^0$ we have 
				$$\tau(K,\theta) \geq \tau(E,\theta).$$
				Explicit calculation (cf., Example 3.1 in \cite{B-IUMJ}) gives that
				$$\lim_{\theta \to (1,0), \ \theta \in \Sigma^0} \tau(E,\theta)= r.$$
				Hence
				$$\liminf_{\theta \to (1,0), \ \theta \in \Sigma^0} \tau(K,\theta)\geq r$$
				for all $r<R$ and hence 
				$$\liminf_{\theta \to (1,0), \ \theta \in \Sigma^0} \tau(K,\theta)\geq R=\lim_{n\to \infty}\tau_{(n,0)}(K).$$
				By rescaling, we may assume $R=1$. Now suppose that 
				$$R=1=\max\{|z|: (z,0) \in K_0\} < \max_{(z,w)\in K}|z|=:|c|.$$
				Let the right-hand-side be attained at the point $(c,d)\in K$. Consider the change of coordinates
				$(z',w'):=(z,w- (d/c)z)$ and let $K'$ be the image of $K$. The Chebyshev constants of $K$ and $K'$ are the same as this transformation preserves the classes of polynomials used to define
				the directional Chebyshev constants. But now for $K'_0:=K'\cap \{w=0\}=\{(z,0)\in K'\}$ we have
				$$|c|=\max\{|z|: (z,0) \in K'_0\}=\max_{(z,w)\in K'}|z|$$
				and we are back in the first case. Thus we have shown that
				$$\liminf_{\theta \to (1,0), \ \theta \in \Sigma^0} \tau(K,\theta)\geq R=\lim_{n\to \infty}\tau_{(n,0)}(K).$$
				From (\ref{liminf}), $\liminf_{\frac{\alpha}{|\alpha|}\to (1,0)}\tau_{\alpha}(K)= \liminf_{\theta \to (1,0), \ \theta \in \Sigma^0} \tau(K,\theta)$ so that 
				$$\liminf_{\frac{\alpha}{|\alpha|}\to (1,0)}\tau_{\alpha}(K)\geq R.$$
				On the other hand, 
				(\ref{Rlim}) gave us 
				$$R=\lim_{n\to \infty}\tau_{(n,0)}(K)\geq \limsup_{\frac{\alpha}{|\alpha|}\to (1,0)}\tau_{\alpha}(K).$$
				Putting these inequalities together gives us that
				$$\lim_{\frac{\alpha}{|\alpha|}\to (1,0)}\tau_{\alpha}(K)$$
				exists and equals $R=\lim_{n\to \infty}\tau_{(n,0)}(K)$.
				
				The continuity of $\theta \to \tau(K,\theta)$ on $\Sigma$ follows immediately from the theorem and the continuity of this function on $\Sigma^0$.
				\end{proof}

			\end{document}